\documentclass[a4paper,11pt, reqno]{amsart}          
\usepackage{amsfonts,amsmath,latexsym,amssymb} 
\usepackage{amsthm}                
\usepackage{mathrsfs,upref}         
\usepackage{mathptmx}               
\usepackage{enumitem}
\usepackage{float}
\usepackage{tikz}
\usepackage{circuitikz}

\newtheorem{theorem}{Theorem}           
\newtheorem{lemma}[theorem]{Lemma} 
                        
\newtheorem{corollary}[theorem]{Corollary}
\newtheorem{claim}[theorem]{Claim}

\theoremstyle{definition}

\newtheorem{definition}[theorem]{Definition}
\newtheorem{example}[theorem]{Example}

\newtheorem{remark}[theorem]{Remark}

\renewcommand{\Re}{\operatorname{Re}}
\renewcommand{\Im}{\operatorname{Im}}
\newcommand{\func}{\mathrm}

\DeclareMathOperator\dist{dist}

\email{anna.muranova@gmail.com}

\thanks{This research was supported by IRTG 2235 Bielefeld-Seoul ``Searching for the regular in the irregular:
Analysis of singular and random systems''.
}

\begin{document}

\author{Anna Muranova}
\address{IRTG 2235, University Bielefeld, Postfach 100131, 33501 Bielefeld,
Germany}
\title{On the effective impedance of finite and infinite networks}
\maketitle

\begin{abstract}
In this paper we deal with the notion of the effective impedance of AC
networks consisting of resistances, coils and capacitors. Mathematically
such a network is a locally finite graph whose edges are endowed with
complex-valued weights depending on a complex parameter $\lambda $ (by the
physical meaning, $\lambda =i\omega $, where $\omega $ is the frequency of
the AC). For finite networks, we prove some estimates of the effective
impedance. Using these estimates, we show that, for infinite networks, the
sequence of impedances of finite graph approximations converges in certain
domains in $\mathbb{C}$ to a holomorphic function of $\lambda$, which allows us to define
the effective impedance of the infinite network.
\end{abstract}

\tableofcontents

\section{Introduction}

Mathematically an electrical network can be represented by a connected graph
whose edges are endowed by weights that are determined by the physical
properties of the connection between two nodes. Here we deal with the
networks consisting of resistors, coils and
capacitors. Assuming that the network is connected to a source of AC with
the frequency $\omega $, each edge $xy$ between the nodes $x,y$ receives
the complex-valued weight (impedance) 
\begin{equation*}
z_{xy}^{\left( \lambda \right)
}=L_{xy}\lambda +R_{xy}+\frac{1}{C_{xy}\lambda },
\end{equation*}
 where $R_{xy}$ is the
resistance of this edge, $L_{xy}$ is the inductance, $C_{xy}$ is the
capacitance, and $\lambda =i\omega .$ The goal is to determine the effective
impedance of the entire network.

If the network is finite then the problem amounts to  a linear system
of Kirchhoff's equations. In absence of coils and capacitors this system has
always non-zero determinant, which implies that the effective impedance is
well-defined (and, of course, is independent of $\lambda $). Note that a
network that consists only of resistances determines naturally a reversible
Markov chain (see e.g. \cite{DS}, \cite{Grimmett}, \cite{LPW}, \cite{Soardi}%
).

For infinite (but locally finite) graphs, again in absence of coils and
capacitors, one constructs first a sequence $\left\{ Z_{n}\right\} $ of 
\emph{partial impedances }that are the effective impedances of an exhaustive
sequence of finite graphs, and then defines the effective impedance $Z$ of
the entire network as the limit $\lim Z_{n}$. This limit always exists due
to the monotonicity of the sequence $\left\{ Z_{n}\right\} $ (cf. \cite{DS}, 
\cite{Grigoryan}, \cite{Grimmett}, \cite{LPW}, \cite{Soardi}, \cite{Woess})

Although the notion of the effective impedance is widely used in physical
and mathematical literature, the problem of justification of this notion in
the presence of coils and capacitors is not satisfactorily solved. In the
case of finite graphs, the determinant of the system may vanish for some
values of $\lambda $ and the system may have infinitely many solutions or no
solutions. Hence, the definition of the effective impedance in this case
requires substantial work that was done in the previous paper of the author 
\cite{Muranova}. In Section \ref{SecFinite} we present an extended version
of these results for finite graphs.

The case of infinite graphs is even more complicated because the sequence $%
\left\{ Z_{n}\right\} $ is complex valued, depends on $\lambda $, and no
monotonicity argument is available.

One of the first examples of computation of effective impedance for an
infinite network was done by Richard Feynman \cite{Feynman}. As it
was observed later (cf. \cite{Enk}, \cite{Klimo}, \cite{UA}, \cite{UY}, \cite{Yoon}), the
sequence $\left\{ Z_{n}\right\} $ of partial effective impedances in this network
(named Feynman's\emph{\ ladder}) converges not for all values of the
frequency $\omega ,$ which raises the question about the validity of
Feynman's computation as well as the problem about a careful mathematical
definition of the effective impedance for infinite networks.

In this paper we make the first attempt to solve this problem. We work with
admittances $\rho _{xy}^{\left( \lambda \right) }=\frac{1}{z_{xy}^{\left(
\lambda \right) }}$ regarded as functions of $\lambda \in \mathbb{C}$
(similarly to \cite{Brune}), and investigate the problem of convergence of
the sequence $\left\{ \mathcal{P}_{n}\left( \lambda \right) \right\} $ of
the partial effective admittances. Our main result, Theorem \ref{convergenceAtInfinity}%
, says that $\mathcal{P}\left( \lambda \right) :=\lim \mathcal{P}_{n}\left(
\lambda \right) $ exists and is a holomorphic function of $\lambda $ in the
domain $\left\{ \Re\lambda >0\right\} $ as well as in some other
regions. In the case of a resistance free network, Corollary \ref{corinf}
says that $\mathcal{P}\left( \lambda \right) $ is holomorphic in $\mathbb{C}%
\setminus \left[ -i\sqrt{S},i\sqrt{S}\right] $ where 
\begin{equation*}
S=\sup_{xy}\frac{1}{C_{xy}L_{xy}}.
\end{equation*}%
These results about infinite networks are proved in Section \ref{SecInfinite}%
. The proof is based on the estimates of admittances for finite networks
that are presented in Section \ref{SecEstimates}. In Section \ref{SecEx} we
give some examples, including a modified Feynman's ladder, that illustrate the
domain of convergence of the sequence $\mathcal{P}_{n}\left( \lambda \right) 
$.

\section{Effective impedance of finite networks}

\label{SecFinite}Let $(V,E)$ be a finite connected graph, where $V$ is the
set of vertices, $\left\vert V\right\vert \geq 2$, and $E$ is the set of
(unoriented) edges.

Assume that each edge $xy\in E$ is equipped with a resistance $R_{xy}$,
inductance $L_{xy}$, and capacitance $C_{xy}$, where $R_{xy},L_{xy}\in
\lbrack 0,+\infty )$ and $C_{xy}\in (0,+\infty ]$, which correspond to the
physical resistor, inductor (coil), and capacitor (see e.g. \cite{Feynman1}%
). It will be convenient to use the inverse capacity 
\begin{equation*}
D_{xy}=\frac{1}{C_{xy}}\in \lbrack 0,+\infty ).
\end{equation*}%
We always assume that for any edge $xy\in E$ 
\begin{equation*}
R_{xy}+L_{xy}+D_{xy}>0.
\end{equation*}%
The \emph{impedance} of the edge $xy$ is defined as the following function
of a complex parameter $\lambda $: 
\begin{equation*}
z_{xy}^{(\lambda )}=R_{xy}+L_{xy}\lambda +\frac{D_{xy}}{\lambda }.
\end{equation*}%
Although the impedance has a physical meaning only for $\lambda =i\omega $,
where $\omega $ is the frequency of the alternating current (that is, $%
\omega $ is a positive real number), it will be convenient to allow $\lambda 
$ to take arbitrary values in $\mathbb{C}\setminus \left\{ 0\right\} $ (cf. 
\cite{Brune}).

In fact, it will be more convenient to work with the \emph{admittance} $\rho
_{xy}^{(\lambda )}$:

\begin{equation}
\rho _{xy}^{(\lambda )}:=\frac{1}{z_{xy}^{(\lambda )}}=\frac{\lambda }{%
L_{xy}\lambda ^{2}+R_{xy}\lambda +{D_{xy}}}.  \label{rhol}
\end{equation}

Define the \emph{physical Laplacian} $\Delta _{\rho }$ as an operator on
functions $f:V\rightarrow \mathbb{C}$ as follows 
\begin{equation}
\Delta _{\rho }f(x)=\sum_{y\in V:y\sim x}(f(y)-f(x))\rho _{xy}^{(\lambda )},
\label{PhysLaplacian}
\end{equation}%
where $x\sim y$ means that $xy\in E$. For convenience let us extend $\rho
_{xy}^{\left( \lambda \right) }$ to all pairs $x,y\in V$ by setting $\rho
_{xy}^{\left( \lambda \right) }=0$ if $x\not\sim y.$ Then the summation in (%
\ref{PhysLaplacian}) can be extended to all $y\in V.$

Let us fix a vertex $a\in V$ and a non-empty subset $B\in V$ such that $%
a\not\in B$. Set $B_{0}=B\cup \{a\}$. The physical meaning of $a$ and $B$ is as follows: $a
$ is the source of AC with the unit voltage, while the set $B$ represents
the ground with zero voltage. We refer to the structure $\Gamma =(V,\rho ,a,B)$ as a
finite \emph{(electrical) network}.

By the complex Ohm's and Kirchhoff's laws, the complex voltage $v:V\rightarrow 
\mathbb{C}$ satisfies the following conditions: 
\begin{equation}
\begin{cases}
\Delta _{\rho }v(x)=0\text{\ \ for }x\in V\setminus B_{0}, \\ 
v(x)=0\text{\ \ for }x\in B, \\ 
v(a)=1.%
\end{cases}
\label{Dirpr}
\end{equation}%
We consider (\ref{Dirpr}) as a discrete boundary value \emph{Dirichlet
problem}.

Denote by $\Lambda $ the set of all those values of $\lambda $ for which $%
\rho _{xy}^{(\lambda )}\in \mathbb{C}\setminus \left\{ 0\right\} $ for all
edges $xy$. The complement $\mathbb{C}\setminus \Lambda $ consists of $%
\lambda =0$ and of all zeros of the equations 
\begin{equation*}
L_{xy}\lambda^{2}+R_{xy}\lambda +{D_{xy}}=0. 
\end{equation*}
In particular, $\mathbb{C}\setminus \Lambda 
$ is a finite set. Clearly, for every $\lambda \in \mathbb{C}\setminus
\Lambda $ we have $\Re\lambda \leq 0$ so that 
\begin{equation*}
\Lambda \supset \{\Re\lambda >0\}.
\end{equation*}%
Observe also that 
\begin{equation}
\Re\lambda >0\Rightarrow \Re z_{xy}^{\left( \lambda \right)
}>0\Rightarrow \Re\rho _{xy}^{(\lambda )}>0  \label{PositivityOfRe}
\end{equation}%
and, for $\lambda \in \Lambda $, 
\begin{equation}
\Re\lambda \geq 0\Rightarrow \Re z_{xy}^{\left( \lambda \right)
}\geq 0\Rightarrow \Re\rho _{xy}^{(\lambda )}\geq 0
\label{NonNegativityOfRe}
\end{equation}%
In what follows we consider the Dirichlet problem (\ref{Dirpr}) only for $%
\lambda \in \Lambda .$

If $v\left( x\right) $ is a solution of (\ref{Dirpr}) then the total current
through $a$ is equal to 
\begin{equation*}
\sum_{x\in V}(1-v(x))\rho _{xa}^{(\lambda )},
\end{equation*}%
which motivates the following definition (cf. \cite{Muranova}).

\begin{definition}
For any $\lambda \in \Lambda $, the \emph{effective admittance} of the
network $\Gamma $ is defined by 
\begin{equation}
\mathcal{P}{(\lambda )}=\sum_{x\in V}(1-v(x))\rho _{xa}^{(\lambda )},
\label{Pdef}
\end{equation}%
where $v$ is a solution of the Dirichlet problem (\ref{Dirpr}). The \emph{%
effective impedance} of $\Gamma $ is defined by 
\begin{equation*}
Z{(\lambda )}=\frac{1}{\mathcal{P}\left( \lambda \right) }=\frac{1}{%
\sum\limits_{x\in V}(1-v(x))\rho _{xa}^{(\lambda )}}.
\end{equation*}%
If the Dirichlet problem (\ref{Dirpr}) has no solution for some $\lambda $,
then we set $\mathcal{P}{(\lambda )}=\infty $ and $Z{(\lambda )}=0$.
\end{definition}

Note that $Z{(\lambda )}$ and $\mathcal{P}{(\lambda )}$ take values in $%
\overline{\mathbb{C}}=\mathbb{C}\cup \{\infty \}$. We will prove below that
in the case when (\ref{Dirpr}) has multiple solution, the values of $Z{%
(\lambda )}$ and $\mathcal{P}{(\lambda )}$ are independent of the choice of
the solution $v$. In the case when $B$ is a singleton, this was proved in 
\cite{Muranova}.

Observe immediately the following symmetry properties that will be used
later on.

\begin{lemma}
$\left( a\right) $ If $\lambda \in \Lambda $ then also $\overline{\lambda }%
\in \Lambda $ and 
\begin{equation}
\mathcal{P}(\overline{\lambda })=\overline{\mathcal{P}({\lambda })}.
\label{Pbar}
\end{equation}

$\left( b\right) $ Assume in addition that $R_{xy}=0$ for all $xy\in E$.
Then $\lambda \in \Lambda $ implies $-\lambda \in \Lambda $ and 
\begin{equation*}
\mathcal{P}(-{\lambda })=-{\mathcal{P}({\lambda })}.
\end{equation*}
\end{lemma}

\begin{proof}
$\left( a\right) $ If $\lambda $ is a root of the equation $L_{xy}\lambda
^{2}+R_{xy}\lambda +D_{xy}=0$ then $\overline{\lambda }$ is also a root,
whence the first claim follows. If $v$ is a solution of (\ref{Dirpr}) for
some $\lambda $ then clearly $\overline{v}$ is a solution of (\ref{Dirpr})
with the parameter $\overline{\lambda }$ instead of $\lambda $. Substituting
into (\ref{Pdef}) and using $\rho ^{\left( \overline{\lambda }\right) }=%
\overline{\rho ^{\left( \lambda \right) }}$, we obtain (\ref{Pbar}).

$\left( b\right) $ The proof is similar to $\left( a\right) $ observing that
if $\lambda $ is a root of $L_{xy}\lambda ^{2}+D_{xy}=0$ then $-\lambda $ is
also a root.
\end{proof}

The following Green's formula was proved in \cite{Muranova} (for simplicity of
notation, we skip the superscript in $\rho ^{\left( \lambda \right) }$ when $%
\lambda $ is fixed).

\begin{lemma}[Green's formula]
For any $\lambda \in \Lambda $ and for any two functions $f,g:V\rightarrow 
\mathbb{C}$ we have 
\begin{equation}
\frac{1}{2}\sum_{x,y\in V}(\nabla _{xy}f){(\nabla _{xy}g)}\rho
_{xy}=-\sum_{x\in V}\Delta _{\rho }f(x){g(x)}=-\sum_{x\in V}\Delta _{\rho
}g(x){f(x)},  \label{Green'sf}
\end{equation}%
where $\nabla _{xy}f=f(y)-f(x)$.
\end{lemma}

\begin{lemma}
For any solution $v$ of (\ref{Dirpr}) we have 
\begin{equation}
\sum_{x\in V}(1-v(x))\rho _{xa}=-\Delta _{\rho }v(a)=\sum_{b\in
B}\Delta _{\rho }v(b)=\frac{1}{2}\sum_{x,y\in V}(\nabla _{xy}v){(\nabla
_{xy}u)}\rho _{xy},  \label{differentEqforP}
\end{equation}%
where $u:V\rightarrow \mathbb{C}$ is any function such that 
\begin{equation}
u(a)=1\text{ \ and\ \ }u\big|_{B}\equiv 0.  \label{u}
\end{equation}
\end{lemma}

\begin{proof}
Using $v(a)=1$, we have 
\begin{equation*}
\Delta _{\rho }v(a)=\sum_{x\in V}(v(x)-v(a))\rho
_{xa}=\sum_{x\in V}(v(x)-1)\rho _{xa},
\end{equation*}%
which proves the first identity in (\ref{differentEqforP}). Since by (\ref%
{Green'sf}) with $f=v$ and $u\equiv 1$ 
\begin{equation*}
\sum_{x\in V}\Delta _{\rho }v(x)=0
\end{equation*}%
and $\Delta _{\rho }v(x)=0$ for all $x\in V\setminus B_{0}$, we obtain 
\begin{equation*}
\sum_{b\in B}\Delta _{\rho }v(b)+\Delta _{\rho }v(a)=0
\end{equation*}%
whence the second identity in (\ref{differentEqforP}) follows.

Finally, to prove the third identity in (\ref{differentEqforP}), we apply
the Green's formula (\ref{Green'sf}) and obtain 
\begin{equation*}
\frac{1}{2}\sum_{x,y\in V}(\nabla _{xy}v){(\nabla _{xy}u)}\rho
_{xy}=-\sum_{x\in V}\Delta _{\rho }v(x)u(x)=-\Delta _{\rho }v(a),
\end{equation*}%
because $\Delta _{\rho }v(x)=0$ for all $x\in V\setminus B_{0}$, while $u%
\big|_{B}\equiv 0$ and $u(a)=1$.
\end{proof}

Comparing (\ref{Pdef}) with (\ref{differentEqforP}), we obtain the identity%
\begin{equation}
\mathcal{P}(\lambda )=\frac{1}{2}\sum_{x,y\in V}(\nabla _{xy}v){(\nabla
_{xy}u)}\rho _{xy}=\sum_{xy\in E}(\nabla _{xy}v){(\nabla _{xy}u)}\rho _{xy},
\label{Puv}
\end{equation}%
where $v$ is a solution of the Dirichlet problem (\ref{Dirpr}) for $\Gamma $
and $u:V\rightarrow \mathbb{C}$ is any function satisfying (\ref{u}).
Choosing here $u=\overline{v}$ we obtain also the identity 
\begin{equation}
\mathcal{P}=\frac{1}{2}\sum_{x,y\in V}|\nabla _{xy}v|^{2}\rho _{xy}
\label{ccp}
\end{equation}%
(conservation of the complex power).

\begin{theorem}
\label{Zeffmult} For any $\lambda \in \Lambda $, the values of $Z{(\lambda )}
$ and $\mathcal{P}{(\lambda )}$ do not depend on the choice of a solution $v$
of the Dirichlet problem (\ref{Dirpr}).
\end{theorem}

\begin{proof}
The proof uses the same argument as in \cite{Muranova}. Let $v_{1}$ and $%
v_{2}$ be two solutions of (\ref{Dirpr}) for the same $\lambda \in \Lambda $%
. By (\ref{differentEqforP}) with $v=v_{1}$ and $u=v_{2}$ we have%
\begin{equation*}
\sum_{x\in V}\left( 1-v_{1}\left( x\right) \right) \rho _{xa}=\frac{1}{2}%
\sum_{x,y\in V}\left( \nabla _{xy}v_{1}\right) \left( \nabla
_{xy}v_{2}\right) \rho _{xy}.
\end{equation*}%
Similarly, we have%
\begin{equation*}
\sum_{x\in V}\left( 1-v_{2}\left( x\right) \right) \rho _{xa}=\frac{1}{2}%
\sum_{x,y\in V}\left( \nabla _{xy}v_{2}\right) \left( \nabla
_{xy}v_{1}\right) \rho _{xy},
\end{equation*}%
whence the identity 
\begin{equation*}
\sum_{x\in V}\left( 1-v_{1}\left( x\right) \right) \rho _{xa}=\sum_{x\in
V}\left( 1-v_{2}\left( x\right) \right) \rho _{xa}
\end{equation*}%
follows. Hence, $v_{1}$ and $v_{2}$ determine the same admittance and
impedance.
\end{proof}

\begin{theorem}
\label{uniqDirpr}The Dirichlet problem (\ref{Dirpr}) has a unique solution $%
v=v^{(\lambda )}$ for all $\lambda \in \Lambda _{0}$ where $\Lambda _{0}$ is
a subset of $\Lambda $ such that $\Lambda \setminus \Lambda _{0}$ is finite.
Besides, $\Lambda _{0}$ contains the domains 
\begin{equation}
\Lambda \cap \{\Re\rho _{xy}^{(\lambda )}>0\ \ \forall xy\in E\},
\label{domainRePos}
\end{equation}
\begin{equation}
\Lambda \cap \{\Im\rho _{xy}^{(\lambda )}>0\ \ \forall xy\in E\}\ \ \ 
\text{and\ \ \ }\Lambda \cap \{\Im\rho _{xy}^{(\lambda )}<0\ \ \forall
xy\in E\}.  \label{domainsIm}
\end{equation}%
Consequently, $\mathcal{P}(\lambda )$ is a rational $\mathbb{C}$-valued
function in $\Lambda _{0}$ and, hence, in any of the domains (\ref%
{domainRePos}) and (\ref{domainsIm}).
\end{theorem}

\begin{proof}
Let us denote the vertices $V\setminus B_{0}$ by $x_{1},\dots ,x_{n}$ and
rewrite the Dirichlet problem (\ref{Dirpr}) as a linear system $n\times n$: 
\begin{equation}
\sum_{j=1}^{n}A_{ij}X_{j}=P_{i}\ \ \text{for any}\ i=1,...,n,  \label{DirprMatrix}
\end{equation}%
where $X_{j}=v(x_{j})$, $P_{i}=\rho _{x_{i}a}$,
\begin{equation*}
A_{ii}=\sum_{y:y\sim x_{i}}\rho _{x_{i}y}\ \text{ and \ }A_{ij}=-\rho
_{x_{i}x_{j}}\ \text{for\ }i\neq j.
\end{equation*}%
Set also 
\begin{equation*}
\mathcal{D}=\det (A_{ij})
\end{equation*}%
and let $\mathcal{D}_{j}$ be the determinant of the matrix obtained by
replacing the column $j$ in the matrix $\{A_{ij}\}$  by the column $\{P_i\}$. Then, by
Cramer's rule, 
\begin{equation*}
X_{j}=\frac{\mathcal{D}_{j}}{\mathcal{D}}
\end{equation*}%
provided $\mathcal{D}\neq 0$. Of course, all these quantities are functions
of $\lambda $. Since all the coefficients $A_{ij}$ and $P_{i}$ are rational
functions of $\lambda $, also $\mathcal{D}=\mathcal{D}(\lambda )$ and $%
\mathcal{D}_{j}=\mathcal{D}_{j}(\lambda )$ are rational functions of $%
\lambda $. For all $\lambda \in \Lambda $ but a finite number, all functions 
$\mathcal{D}_{j}(\lambda )$ and $\mathcal{D}(\lambda )$ take values in $%
\mathbb{C}$. The existence and uniqueness of a solution  of (\ref{DirprMatrix}) is equivalent to $%
\mathcal{D}(\lambda )\neq 0$. Hence, define $\Lambda _{0}$ as the subset of $%
\Lambda $ where all functions $\mathcal{D}_{j}(\lambda )$ and $\mathcal{D}%
(\lambda )$ take values in $\mathbb{C}$ and, besides, $\mathcal{D}\left(
\lambda \right) \neq 0.$ Since $\mathcal{D}(\lambda )$ is a rational
function on $\lambda $, it may have only finitely many zeros or  vanish identically. 

Hence, it suffices to exclude the latter case, that is, to show that $%
\Lambda _{0}\neq \emptyset $. For that, let us prove that $\Lambda _{0}$
contains the domain (\ref{domainRePos}) that in turn, by (\ref%
{PositivityOfRe}) and (\ref{NonNegativityOfRe}), contains $\left\{ \Re%
\lambda >0\right\} $ and, hence, is non-empty. In order to show that $%
\mathcal{D}\left( \lambda \right) \neq 0$ for any $\lambda $ from (\ref%
{domainRePos}), it suffices to verify that the homogeneous Dirichlet problem 
\begin{equation}
\begin{cases}
\Delta _{\rho }u(x)=0\mbox { on }V\setminus B_{0}, \\ 
u(x)=0\mbox { on }B_{0}%
\end{cases}
\label{DirHom}
\end{equation}%
has a unique solution $u\equiv 0$. Indeed, by Green's formula we have 
\begin{align*}
\sum_{xy\in E}\left\vert \nabla _{xy}u\right\vert ^{2}\rho _{xy}=& \frac{1}{2%
}\sum_{x,y\in V}\left\vert \nabla _{xy}u\right\vert ^{2}\rho
_{xy}=-\sum_{x,y\in V}\Delta _{\rho }u(x)\overline{u}(x) \\
=& -\sum_{x\in V\setminus B_{0}}\Delta _{\rho }u(x)\overline{u}%
(x)-\sum_{x\in B_{0}}\Delta _{\rho }u(x)\overline{u}(x)=0,
\end{align*}%
since $u$ is a solution of (\ref{DirHom}). Since $\Re\rho _{xy}>0$, we
conclude that $\left\vert \nabla _{xy}u\right\vert =0$ on all the edges. By
the connectedness of the graph this implies that $u=\func{const}$. Since $u%
\big|_{B_{0}}\equiv 0$, we conclude that $u\equiv 0$.

In the same way the domains (\ref{domainsIm}) are subsets of $\Lambda _{0}.$

Finally, by the above argument, $v^{\left( \lambda \right) }\left( x\right) $
is a rational function of $\lambda $, so that the last claim follows from (%
\ref{Pdef}).
\end{proof}

\begin{remark}
\label{Rege0}Since $\{\Re\lambda >0\}$ is contained in $\Lambda _{0}$,
we see that $\mathcal{P}(\lambda )$ is a holomorphic function in $\left\{ 
\Re\lambda >0\right\} $. If $R_{xy}>0$ for all $xy\in E$, then also $%
\Lambda \cap \{\Re\lambda \geq 0\}$ is a subset of (\ref{domainRePos}).
\end{remark}

\begin{remark}
The uniqueness of the solution of the Dirichlet problem for the domain (\ref%
{domainRePos}) follows also from \cite[Lemma 4.4]{Vasquez} and (\ref%
{PositivityOfRe}).
\end{remark}

\begin{example}
Let us consider the finite network as at  Fig. \ref{finiteNetwork}, where all inductances, capacitances  and resistance are equal to $1$, with $a=\{1\}$, $B=\{0\}$. Then $\Lambda=\mathbb{C}\setminus\{0\}$.
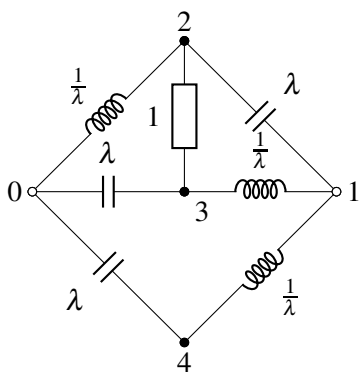
\begin{figure}[H]
\begin{circuitikz} 

\ctikzset{resistor = european}
\ctikzset{label/align = straight}
  \draw

 (0,0) node[anchor=east]{$0$}
  (0,0) to[/tikz/circuitikz/bipoles/length=20pt,C=${\lambda}$, o-*] (2,0) 
  (2,0) node[anchor=north west]{$3$}
  (2,0) to[/tikz/circuitikz/bipoles/length=30pt,L=$\frac{1}{\lambda}$,  *-o] (4,0)
  (4,0) node[anchor=west]{$1$}
  (2,-2) node[anchor=north]{$4$}
  (2,2) node[anchor=south]{$2$}
  (0,0) to[/tikz/circuitikz/bipoles/length=30pt,L=$\frac{1}{\lambda}$,   o-*] (2,2)
  (2,0) to[/tikz/circuitikz/bipoles/length=30pt, R=$1$,   *-*] (2,2)
  (2,2) to[/tikz/circuitikz/bipoles/length=20pt,C=${\lambda}$,   *-o] (4,0)
  (2,-2) to[/tikz/circuitikz/bipoles/length=20pt,C=${\lambda}$,   *-o] (0,0)
  (4,0) to[/tikz/circuitikz/bipoles/length=30pt,L=$\frac{1}{\lambda}$,   o-*] (2,-2);
 \end{circuitikz}
\caption{Example of a finite network}
\label{finiteNetwork}
\end{figure}
The effective admittance of this network is
calculated in \cite[Example 28]{Muranova}. We have 
\begin{equation*}
\mathcal{P}(\lambda)=
\begin{cases}
\infty,\ \ \ \lambda=\pm i,\\
-\frac{3}{2}, \ \lambda=-1,\\
\frac{\lambda^2+\lambda+1}{\lambda^2+1},\  \mbox {otherwise},
\end{cases}
\end{equation*}
and $\Lambda_0=\mathbb{C}\setminus\{-1,0,\pm i\}$.
\end{example}

Our next goal is to define the effective admittance of infinite networks. We
will do it in Section \ref{SecInfinite}, but before that we prove some
estimates for $\mathcal{P}({\lambda })$ on finite networks.

\section{Estimates of the effective admittance of finite networks}

\label{SecEstimates}We use the same setup and notation as in Section \ref%
{SecFinite}.
We skip $\lambda$ in notations $\rho^{(\lambda)}_{xy}$ and $\mathcal{P}(\lambda)$  when the value of $\lambda$  is fixed.
\subsection{An upper bound of the admittance using $\Re \protect%
\lambda $}

\begin{theorem}
\label{theoremInfReepsilon} Let $\lambda \in \Lambda $ be fixed and assume
that, for some $\epsilon >0$,%
\begin{equation}
\inf_{xy\in E}\frac{\Re\rho _{xy}}{\left\vert \rho _{xy}\right\vert }%
\geq \epsilon .  \label{infReepsilon}
\end{equation}%
Then 
\begin{equation}
\left\vert \mathcal{P}\right\vert \leq \frac{1}{\epsilon ^{2}}\sum_{x\sim
a}\left\vert \rho _{xa}\right\vert .  \label{boundPepsilon}
\end{equation}%
The same result is true if one assumes instead of (\ref{infReepsilon}) that 
\begin{equation*}
\inf_{xy\in E}\frac{\Im\rho _{xy}}{\left\vert \rho _{xy}\right\vert }%
\geq \epsilon
\end{equation*}%
or 
\begin{equation*}
\inf_{xy\in E}\frac{-\Im\rho _{xy}}{\left\vert \rho _{xy}\right\vert }%
\geq \epsilon .
\end{equation*}
\end{theorem}

\begin{proof}
Under the hypothesis (\ref{infReepsilon}) the Dirichlet problem (\ref{Dirpr}%
) has by Theorem \ref{uniqDirpr} a unique solution $v$. We have by (\ref{ccp}%
) 
\begin{equation}
\left\vert \mathcal{P}\right\vert \geq \Re\mathcal{P}=\sum_{xy\in
E}\left\vert \nabla _{xy}v\right\vert ^{2}\Re\rho _{xy}\geq \epsilon
\sum_{xy\in E}\left\vert \nabla _{xy}v\right\vert ^{2}\left\vert \rho
_{xy}\right\vert .  \label{Pbelowepsilon}
\end{equation}%
Applying (\ref{Puv}) with the function 
\begin{equation*}
u=\mathbf{1}_{\{a\}}
\end{equation*}
and using the inequality 
\begin{equation*}
2|ab|\leq {\epsilon }|a|^{2}+\frac{1}{\epsilon }|b|^{2},
\end{equation*}%
we obtain 
\begin{equation*}
\left\vert \mathcal{P}\right\vert \leq \sum_{xy\in E}\left\vert \nabla
_{xy}v\right\vert \left\vert \nabla _{xy}u\right\vert \left\vert \rho
_{xy}\right\vert \leq \frac{\epsilon }{2}\sum_{xy\in E}|\nabla
_{xy}v|^{2}\left\vert \rho _{xy}\right\vert +\frac{1}{2\epsilon }\sum_{xy\in
E}|\nabla _{xy}u|^{2}\left\vert \rho _{xy}\right\vert .
\end{equation*}%
Setting 
\begin{equation*}
U:=\sum_{xy\in E}\left\vert \nabla _{xy}u\right\vert ^{2}\left\vert \rho
_{xy}\right\vert =\sum_{x\sim a}\left\vert \rho _{xa}\right\vert ,
\end{equation*}%
we obtain 
\begin{equation*}
\left\vert \mathcal{P}\right\vert \leq \frac{\epsilon }{2}\sum_{xy\in
E}|\nabla _{xy}v|^{2}\left\vert \rho _{xy}\right\vert +\frac{1}{2\epsilon }U.
\end{equation*}%
Combing this with (\ref{Pbelowepsilon}) yields 
\begin{equation*}
\frac{\epsilon }{2}\sum_{xy\in E}|\nabla _{xy}v|^{2}\left\vert \rho
_{xy}\right\vert \leq \frac{1}{2\epsilon }U
\end{equation*}%
whence by (\ref{ccp}) 
\begin{equation*}
\left\vert \mathcal{P}\right\vert =\left\vert \sum_{xy\in E}\left\vert
\nabla _{xy}v\right\vert ^{2}\rho _{xy}\right\vert \leq \sum_{xy\in
E}\left\vert \nabla _{xy}v\right\vert ^{2}\left\vert \rho _{xy}\right\vert
\leq \frac{1}{\epsilon ^{2}}U.
\end{equation*}

The conditions with $\Im\rho _{xy}$ are handled in the same way.
\end{proof}

In order to be able to verify (\ref{infReepsilon}), we need the following
lemma.

\begin{lemma}
Let $L$, $R$, $D$ be non-negative real numbers and $\lambda \in \mathbb{C}%
\setminus \{0\}$. If 
\begin{equation*}
z:=R+L\lambda +\frac{D}{\lambda }\neq 0
\end{equation*}%
then 
\begin{equation*}
\frac{\Re z}{|z|}\geq \frac{\Re\lambda }{|\lambda |}
\end{equation*}%
and, for $\rho =\frac{1}{z}$, 
\begin{equation*}
\frac{\Re\rho }{|\rho |}\geq \frac{\Re\lambda }{|\lambda |}.
\end{equation*}
\end{lemma}

\begin{proof}
We have 
\begin{equation*}
\Re z=R+L\Re\lambda +\frac{D\Re\lambda }{|\lambda |^{2}}%
\geq \left( R+L|\lambda |+\frac{D}{|\lambda |}\right) \frac{\Re\lambda 
}{|\lambda |}
\end{equation*}%
and 
\begin{equation*}
|z|\leq R+L|\lambda |+\frac{D}{|\lambda |},
\end{equation*}%
whence 
\begin{equation*}
\frac{\Re z}{|z|}\geq \frac{\Re\lambda }{|\lambda |}.
\end{equation*}%
Finally, we have 
\begin{equation*}
\rho =\frac{1}{z}=\frac{\Re z-i\Im z}{|z|^{2}}
\end{equation*}%
and, hence, 
\begin{equation*}
\frac{\Re\rho }{|\rho |}=\frac{\Re z}{|z|^{2}}|z|=\frac{\Re z%
}{|z|}\geq \frac{\Re\lambda }{|\lambda |}.
\end{equation*}
\end{proof}

\begin{corollary}
If $\Re\lambda >0$ then 
\begin{equation}
\left\vert \mathcal{P}(\lambda )\right\vert \leq \frac{|\lambda |^{2}}{%
\left( \Re\lambda \right) ^{2}}\sum_{x\sim a}\left\vert \rho
_{xa}^{(\lambda )}\right\vert .  \label{boundPlambda}
\end{equation}
\end{corollary}

\begin{proof}
Indeed, we have for all $xy\in E$ 
\begin{equation*}
\frac{\Re\rho _{xy}^{(\lambda )}}{|\rho _{xy}^{(\lambda )}|}\geq \frac{%
\Re\lambda }{|\lambda |}=:\epsilon .
\end{equation*}%
Substituting into (\ref{boundPepsilon}) we obtain (\ref{boundPlambda}).
\end{proof}

\begin{corollary}
\label{corollaryBoundPlambdaC} If $\Re\lambda >0$ then 
\begin{equation}
\left\vert \mathcal{P}(\lambda )\right\vert \leq C\frac{|\lambda |^{2}\left(
1+|\lambda |^{2}\right) }{\left( \Re\lambda \right) ^{3}},
\label{boundPlambdaC}
\end{equation}%
where 
\begin{equation*}
C=\sum_{x\sim a}\frac{1}{R_{xa}+L_{xa}+D_{xa}}.
\end{equation*}
\end{corollary}

\begin{proof}
We have for $z=R+L\lambda +\frac{D}{\lambda }$ 
\begin{align*}
|z|\geq & \Re z\geq R+L\Re\lambda +\frac{D\Re\lambda }{%
|\lambda |^{2}} \\
\geq & (R+L+D)\min \left( 1,\Re\lambda ,\frac{\Re\lambda }{%
|\lambda |^{2}}\right) \\
\geq & (R+L+D)\min \left( \Re\lambda ,\frac{\Re\lambda }{%
|\lambda |^{2}}\right) \\
=& (R+L+D)\Re\lambda \min \left( 1,{|\lambda |^{-2}}\right) \\
\geq & (R+L+D)\Re\lambda \frac{1}{1+{|\lambda |^{2}}},
\end{align*}%
since $\frac{\left( \Re\lambda \right) ^{2}}{|\lambda |^{2}}\leq 1$
and, therefore, either $\Re\lambda \leq 1$ or $\frac{\Re\lambda 
}{|\lambda |^{2}}\leq 1$. It follows that, for $\rho =\frac{1}{z}$, 
\begin{equation*}
|\rho |\leq \frac{1}{R+L+D}\frac{1+{|\lambda |^{2}}}{\Re\lambda }.
\end{equation*}%
Hence, 
\begin{equation*}
\sum_{x\sim a}\left\vert \rho _{xa}(\lambda )\right\vert \leq \frac{%
1+|\lambda |^{2}}{\Re\lambda }\sum_{x\sim a}\frac{1}{%
R_{xa}+L_{xa}+D_{xa}}.
\end{equation*}%
Substituting into (\ref{boundPlambda}) we obtain (\ref{boundPlambdaC}).
\end{proof}

\begin{remark}
In the case $L>0$ we can use in the domain $\{\Re\lambda >0\}$ the
estimate $|z|\geq L\Re\lambda $ which implies 
\begin{equation*}
|\rho |\leq \frac{1}{L\Re\lambda }.
\end{equation*}%
Hence, if $L_{xa}>0$ for all $x\sim a$ then 
\begin{equation*}
\sum_{x\sim a}\left\vert \rho _{xa}^{(\lambda )}\right\vert \leq
C^{\prime }\frac{1}{\Re\lambda },
\end{equation*}%
where 
\begin{equation*}
C^{\prime }=\sum_{x\sim a}\frac{1}{L_{xa}}.
\end{equation*}
\end{remark}

Therefore, by (\ref{boundPlambda}) in this case, in the domain $\{\Re%
\lambda >0\}$ we have 
\begin{equation*}
\left\vert \mathcal{P}(\lambda )\right\vert \leq \frac{C^{\prime }|\lambda
|^{2}}{\left( \Re\lambda \right) ^{3}}.
\end{equation*}

\subsection{An upper bound of the admittance using large $\Im \protect%
\lambda$}

\begin{lemma}
\label{lemmaLargeIm} Let $R$, $L$, $D$ be non-negative numbers. Let $L>0$
and $\lambda \in \mathbb{C}$ be such that 
\begin{equation*}
\Im\lambda >0,\ \ |\lambda |^{2}>\frac{D}{L}.
\end{equation*}%
Then 
\begin{equation*}
z:=R+L\lambda +\frac{D}{\lambda }\neq 0
\end{equation*}%
and for $\rho =\frac{1}{z}$ we have 
\begin{equation}
-\frac{\Im\rho }{|\rho |}\geq \frac{1-\frac{D}{L|\lambda |^{2}}}{%
|\lambda |+\frac{D}{L|\lambda |}+\frac{R}{L}}\Im\lambda
\label{largeIm}
\end{equation}%
and 
\begin{equation*}
|\rho |\leq \frac{1}{\left( L-\frac{D}{|\lambda |^{2}}\right) \Im%
\lambda }.
\end{equation*}
\end{lemma}

\begin{proof}
We have 
\begin{equation*}
\Im z=L\Im \lambda-\frac{D\Im \lambda}{|\lambda|^2}%
=\left(L-\frac{D}{|\lambda|^2}\right)\Im \lambda>0.
\end{equation*}
In particular, $z\ne 0$. We have also 
\begin{equation*}
|z|\le R+L|\lambda|+\frac{D}{|\lambda|}.
\end{equation*}
It follows that 
\begin{equation*}
-\frac{\Im \rho}{|\rho|}=\frac{\Im z}{|z|}\ge \frac{\left(L-%
\frac{D}{|\lambda|^2}\right)\Im \lambda}{R+L|\lambda|+\frac{D}{%
|\lambda|}}\ge \frac{1-\frac{D}{L|\lambda|^2}}{|\lambda|+\frac{D}{L|\lambda|}%
+\frac{R}{L}}\Im \lambda
\end{equation*}
which proves (\ref{largeIm}). Finally, we have 
\begin{equation*}
|\rho|=\frac{1}{|z|}\le \frac{1}{\Im z}= \frac{1}{\left(L-\frac{D}{%
|\lambda|^2}\right)\Im \lambda}.
\end{equation*}
\end{proof}

\begin{theorem}
\label{HolomorphicDomainLargeIm} Assume that $L_{xy}>0$ for all $xy\in E$.
Set 
\begin{equation*}
S_{D}=\sup_{xy\in E}\frac{D_{xy}}{L_{xy}},\text{\ \ \ \ }\ S_{R}=\sup_{xy\in E}%
\frac{R_{xy}}{L_{xy}},
\end{equation*}%
and 
\begin{equation*}
C^{\prime }=\sum_{x\sim a}\frac{1}{L_{xa}}.
\end{equation*}%
Then, in the domain, 
\begin{equation}
\Omega =\{\lambda \in \mathbb{C}\mid \Im\lambda \neq 0\ \ \text{and\ \ 
}|\lambda |^{2}>S_{D}\}
\end{equation}%
the function $\mathcal{P}(\lambda )$ is holomorphic and 
\begin{equation*}
\left\vert \mathcal{P}(\lambda )\right\vert \leq \frac{C^{\prime }(2|\lambda
|+S_{R})^{2}|\lambda |^{6}}{(|\lambda |^{2}-S_{D})^{3}|\Im\lambda |^{3}%
}.
\end{equation*}
\end{theorem}

\begin{figure}[H]
\begin{tikzpicture}[scale=0.7]
    \begin{scope}[thick,font=\scriptsize]


    \end{scope}
   \path [draw=none,fill=gray,semitransparent] (-5,-5) rectangle (5,5);
    \draw [dashed,fill=white] (0,0) circle (3);
   \path [draw=none,fill=white] (-5,-0.05) rectangle (5,0.05);
    \draw (3,-3pt) -- (3,3pt)   node [above right] {$\sqrt{S_D}$};
    \draw (-3pt,3) -- (3pt,3)   node [above right] {$i\sqrt{S_D}$};
    \draw [->] (-5,0) -- (5,0) node [below left]  {$\Re \lambda$};
    \draw [->] (0,-5) -- (0,5) node [below right] {$\Im \lambda$};
    \draw [] (-5,0) -- (-2,0);
    \node [below right,gray] at (+3,+3.5) {$\Omega$};
\end{tikzpicture}
\caption{The domain $\Omega =\{\lambda \in \mathbb{C}\mid \Im\lambda \neq 0\ \ \text{and\ \ 
}|\lambda |^{2}>S_{D}\}$.}
\label{figOmega}
\end{figure}
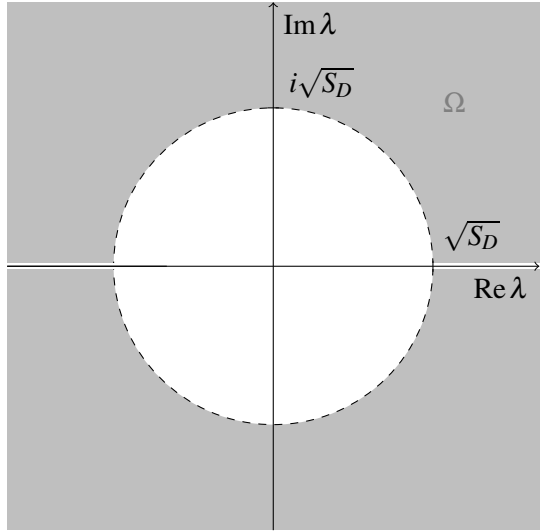

\begin{proof}
By the symmetry $\lambda \rightarrow \overline{\lambda }$, it suffices to
prove the both claims in the domain 
\begin{equation*}
\Omega _{+}=\{\lambda \in \mathbb{C}\mid \Im\lambda >0\ \text{and\ \ }%
|\lambda |^{2}>S_{D}\}.
\end{equation*}%
For $\lambda \in \Omega _{+}$ we have by Lemma \ref{lemmaLargeIm} that $%
z_{xy}\neq 0$, whence $\lambda \in \Lambda $. By Lemma \ref{lemmaLargeIm} we
have for all $xy\in E$ and $\lambda \in \Omega _{+}$ 
\begin{equation}
-\frac{\Im\rho _{xy}}{|\rho _{xy}|}\geq \frac{1-\frac{D_{xy}}{%
L_{xy}|\lambda |^{2}}}{|\lambda |+\frac{D_{xy}}{L_{xy}|\lambda |}+\frac{%
R_{xy}}{L_{xy}}}\Im\lambda \geq \frac{1-\frac{S_{D}}{|\lambda |^{2}}}{%
|\lambda |+\frac{S_{D}}{|\lambda |}+S_{R}}\Im\lambda \geq \frac{1-%
\frac{S_{D}}{|\lambda |^{2}}}{2|\lambda |+S_{R}}\Im\lambda >0,
\label{SdSreq1}
\end{equation}%
since $\frac{S_{D}}{|\lambda |}<|\lambda |$. By Theorem \ref{uniqDirpr} we
conclude that $\mathcal{P}(\lambda )$ is a holomorphic function in $\Omega
_{+}$.

Using (\ref{SdSreq1}), we obtain by Theorem \ref{theoremInfReepsilon} that
for $\lambda \in \Omega _{+}$ 
\begin{equation}
\left\vert \mathcal{P}(\lambda )\right\vert \leq \left( \frac{2|\lambda
|+S_{R}}{\left( 1-\frac{S_{D}}{|\lambda |^{2}}\right) \Im\lambda }%
\right) ^{2}\sum_{x\sim a}\left\vert \rho _{xa}^{(\lambda
)}\right\vert .  \label{SdSreq2}
\end{equation}%
Next, we have by Lemma \ref{lemmaLargeIm} 
\begin{equation*}
|\rho _{xy}|\leq \frac{1}{\left( L_{xy}-\frac{D_{xy}}{|\lambda |^{2}}\right) 
\Im\lambda }\leq \frac{1}{L_{xy}\left( 1-\frac{S_{D}}{|\lambda |^{2}}%
\right) \Im\lambda },
\end{equation*}%
whence 
\begin{equation*}
\sum_{x\sim a}\left\vert \rho _{xa}\right\vert \leq \frac{%
\sum_{x\sim a}\left( L_{xa}\right) ^{-1}}{\left( 1-\frac{S_{D}}{%
|\lambda |^{2}}\right) \Im\lambda }=\frac{C^{\prime }}{\left( 1-\frac{%
S_{D}}{|\lambda |^{2}}\right) \Im\lambda }.
\end{equation*}%
It follows from (\ref{SdSreq2}) that 
\begin{equation*}
\left\vert \mathcal{P}(\lambda )\right\vert \leq \left( \frac{2|\lambda
|+S_{R}}{\left( 1-\frac{S_{D}}{|\lambda |^{2}}\right) \Im\lambda }%
\right) ^{2}\frac{C^{\prime }}{\left( 1-\frac{S_{D}}{|\lambda |^{2}}\right) 
\Im\lambda }=\frac{C^{\prime }(2|\lambda |+S_{R})^{2}}{\left( 1-\frac{%
S_{D}}{|\lambda |^{2}}\right) ^{3}\left( \Im\lambda \right) ^{3}}
\end{equation*}%
which was to be proved.
\end{proof}

\begin{corollary}
Under the hypothesis of Theorem \ref{HolomorphicDomainLargeIm}, assume in
addition that $R_{xy}=0$ for all $xy\in E$. Then $\mathcal{P}(\lambda )$ is
holomorphic in $\mathbb{C}\setminus J$ where 
\begin{equation*}
J=\left[ -i\sqrt{S_{D}},i\sqrt{S_{D}}\right] .
\end{equation*}
\end{corollary}

\begin{proof}
In this case we have the symmetry $\mathcal{P}(-\lambda )=-\mathcal{P}%
(\lambda )$. By Theorem \ref{uniqDirpr} (see Remark \ref{Rege0}) and Theorem %
\ref{HolomorphicDomainLargeIm}, $\mathcal{P}(\lambda )$ is holomorphic in the
union 
\begin{equation*}
\{\Re\lambda \neq 0\}\cup \{\left( \Im\lambda \right)
^{2}>S_{D}\},
\end{equation*}%
that coincides with $\mathbb{C}\setminus J$.
\end{proof}

\subsection{An upper bound of the admittance using small $\Im \protect%
\lambda$}

\begin{lemma}
\label{lemmaSmallIm} Let $R$, $L$, $D$ be non-negative numbers. Let $L>0$
and $\lambda \in \mathbb{C}$ be such that 
\begin{equation*}
\Im\lambda >0\ \text{\ and}\ \ \left\vert \lambda \right\vert ^{2}<%
\frac{D}{L}.
\end{equation*}%
Then 
\begin{equation*}
z:=R+L\lambda +\frac{D}{\lambda }\neq 0
\end{equation*}%
and, for $\rho =\frac{1}{z}$, we have 
\begin{equation}
\frac{\Im\rho }{|\rho |}\geq \frac{\frac{D}{L|\lambda |^{2}}-1}{%
|\lambda |+\frac{D}{L|\lambda |}+\frac{R}{L}}\Im\lambda
\label{smallIm}
\end{equation}%
and 
\begin{equation*}
|\rho |\leq \frac{1}{\left( \frac{D}{|\lambda |^{2}}-L\right) \Im%
\lambda }.
\end{equation*}
\end{lemma}

\begin{proof}
We have 
\begin{equation*}
-\Im z=-L\Im \lambda+\frac{D\Im \lambda}{|\lambda|^2}%
=\left(\frac{D}{|\lambda|^2}-L\right)\Im \lambda>0.
\end{equation*}
In particular, $z\ne 0$. We have also 
\begin{equation*}
|z|\le R+L|\lambda|+\frac{D}{|\lambda|}.
\end{equation*}
It follows that 
\begin{equation*}
\frac{\Im \rho}{|\rho|}=-\frac{\Im z}{|z|}\ge \frac{\left(\frac{D%
}{|\lambda|^2}-L\right)\Im \lambda}{R+L|\lambda|+\frac{D}{|\lambda|}}%
\ge \frac{\frac{D}{L|\lambda|^2}-1}{|\lambda|+\frac{D}{L|\lambda|}+\frac{R}{L%
}}\Im \lambda
\end{equation*}
which proves (\ref{smallIm}). Finally, we have 
\begin{equation*}
|\rho|=\frac{1}{|z|}\le \frac{1}{-\Im z}=\frac{1}{\left(\frac{D}{%
|\lambda|^2}-L\right)\Im \lambda}.
\end{equation*}
\end{proof}

\begin{theorem}
\label{HolomorphicDomainSmallIm} Assume that $L_{xy}>0$ for all $xy\in E$.
Set 
\begin{equation*}
S_{D}=\sup_{xy\in E}\frac{D_{xy}}{L_{xy}},\ \ S_{D}^{\ast }=\inf_{xy\in E}%
\frac{D_{xy}}{L_{xy}},\ \ \ S_{R}=\sup_{xy\in E}\frac{R_{xy}}{L_{xy}},
\end{equation*}%
and 
\begin{equation*}
C^{\prime }=\sum_{x\sim a}\frac{1}{L_{xa}}.
\end{equation*}%
Then in the domain 
\begin{equation}
\Omega ^{\ast }=\{\lambda \in \mathbb{C}\mid \Im\lambda \neq 0\ \ 
\text{and\ \ }|\lambda |^{2}<S_{D}^{\ast }\}
\end{equation}%
the function $\mathcal{P}(\lambda )$ is holomorphic and 
\begin{equation}
\left\vert \mathcal{P}(\lambda )\right\vert \leq \frac{C^{\prime }(|\lambda
|^{2}+S_{R}|\lambda |+S_{D})^{2}|\lambda |^{4}}{(S_{D}^{\ast }-|\lambda
|^{2})^{3}|\Im\lambda |^{3}}.  \label{RhoinOmegaStar}
\end{equation}
\end{theorem}

\begin{figure}[H]
\begin{tikzpicture}[scale=0.7]
    \begin{scope}[thick,font=\scriptsize]


    \end{scope}
   \draw [dashed,fill=gray,semitransparent] (0,0) circle (2);
    \draw [dashed,fill=none] (0,0) circle (2);
   \path [draw=none,fill=white] (-5,-0.05) rectangle (5,0.05);
    \draw (2,-3pt) -- (2,3pt)   node [above right] {$\sqrt{S^{\ast }_D}$};
    \draw (-3pt,2) -- (3pt,2)   node [above right] {$i\sqrt{S^{\ast }_D}$};
    \draw [->] (-5,0) -- (5,0) node [below left]  {$\Re \lambda$};
    \draw [->] (0,-5) -- (0,5) node [below right] {$\Im \lambda$};
    \draw [] (-5,0) -- (-2,0);
    \node [below right,gray] at (0.5,1.5) {$\Omega^*$};
\end{tikzpicture}
\caption{The domain $\Omega^{\ast } =\{\lambda \in \mathbb{C}\mid \Im\lambda \neq 0\ \ \text{and\ \ 
}|\lambda |^{2}<S^{\ast }_{D}\}$.}
\label{figOmegaAst}
\end{figure}

\begin{proof}
If $S_{D}^{\ast }=0$ then $\Omega ^{\ast }=\emptyset $ and there is nothing
to prove. Let $S_{D}^{\ast }>0$. By the symmetry $\lambda \rightarrow 
\overline{\lambda }$, it suffices to prove the both claims in the domain 
\begin{equation*}
\Omega _{+}^{\ast }=\{\lambda \in \mathbb{C}\mid \Im\lambda >0%
\mbox{
and }|\lambda |^{2}<S_{D}^{\ast }\}.
\end{equation*}%
For $\lambda \in \Omega _{+}^{\ast }$ we have by Lemma \ref{lemmaSmallIm}
that $z_{xy}\neq 0$, whence $\lambda \in \Lambda $. By Lemma \ref%
{lemmaSmallIm} we have for all $xy\in E$ and $\lambda \in \Omega _{+}^{\ast
} $ 
\begin{equation}
\frac{\Im\rho _{xy}}{|\rho _{xy}|}\geq \frac{\frac{D_{xy}}{%
L_{xy}|\lambda |^{2}}-1}{|\lambda |+\frac{D_{xy}}{L_{xy}|\lambda |}+\frac{%
R_{xy}}{L_{xy}}}\Im\lambda \geq \frac{\frac{S_{D}^{\ast }}{|\lambda
|^{2}}-1}{|\lambda |+\frac{S_{D}}{|\lambda |}+S_{R}}\Im\lambda >0,
\label{SdSrSeq1}
\end{equation}%
By Theorem \ref{uniqDirpr} we conclude that $\mathcal{P}(\lambda )$ is a
holomorphic function in $\Omega _{+}^{\ast }$.

Using (\ref{SdSrSeq1}), we obtain by Theorem \ref{theoremInfReepsilon} that
for all $\lambda \in \Omega _{+}^{\ast }$ 
\begin{equation*}
\left\vert \mathcal{P}(\lambda )\right\vert \leq \left( \frac{|\lambda |+%
\frac{S_{D}}{|\lambda |}+S_{R}}{\left( \frac{S_{D}^{\ast }}{|\lambda |^{2}}%
-1\right) \Im\lambda }\right) ^{2}\sum_{x\sim a}\left\vert \rho
_{xa}^{(\lambda )}\right\vert .
\end{equation*}%
Next, we have by Lemma \ref{lemmaSmallIm} 
\begin{equation*}
|\rho _{xy}|\leq \frac{1}{\left( \frac{D_{xy}}{|\lambda |^{2}}-L_{xy}\right) 
\Im\lambda }=\frac{1}{L_{xy}\left( \frac{D_{xy}}{L_{xy}|\lambda |^{2}}%
-1\right) \Im\lambda }\leq \frac{1}{L_{xy}\left( \frac{S_{D}^{\ast }}{%
|\lambda |^{2}}-1\right) \Im\lambda },
\end{equation*}%
whence 
\begin{equation*}
\sum_{x\sim a}\left\vert \rho _{xa}\right\vert \leq \frac{%
\sum_{x\sim a}\left( L_{xa}\right) ^{-1}}{\left( \frac{S_{D}^{\ast }%
}{|\lambda |^{2}}-1\right) \Im\lambda }=\frac{C^{\prime }}{\left( 
\frac{S_{D}^{\ast }}{|\lambda |^{2}}-1\right) \Im\lambda }.
\end{equation*}%
It follows that 
\begin{equation*}
\left\vert \mathcal{P}(\lambda )\right\vert \leq \frac{C^{\prime }\left(
|\lambda |+\frac{S_{D}}{|\lambda |}+S_{R}\right) ^{2}}{\left( \frac{%
S_{D}^{\ast }}{|\lambda |^{2}}-1\right) ^{3}(\Im\lambda )^{3}}
\end{equation*}%
whence (\ref{RhoinOmegaStar}) follows.
\end{proof}

\begin{corollary}
Under the hypothesis of Theorem \ref{HolomorphicDomainSmallIm}, assume in
addition that $R_{xy}=0$ for all $xy\in E$. Then $\mathcal{P}(\lambda )$ is
holomorphic in the domain $\mathbb{C}\setminus J^{\ast }$ where 
\begin{equation*}
J^{\ast }=\left[ i\sqrt{S_{D}^{\ast }},i\sqrt{S_{D}}\right] \cup \left[ -i%
\sqrt{S_{D}},i\sqrt{S_{D}^{\ast }}\right] \cup \{0\}.
\end{equation*}
\end{corollary}

\begin{proof}
In this case we have the symmetry $\mathcal{P}(-\lambda )=-\mathcal{P}%
(\lambda )$. By Theorem \ref{uniqDirpr} (see Remark \ref{Rege0}), Theorems %
\ref{HolomorphicDomainLargeIm} and \ref{HolomorphicDomainSmallIm} $\mathcal{P%
}(\lambda )$ is holomorphic in the union 
\begin{equation*}
\{\Re\lambda \neq 0\}\cup \{(\Im\lambda )^{2}>S_{D}\}\cup \{0<(%
\Im\lambda )^{2}<S_{D}^{\ast }\}
\end{equation*}%
that coincides with $\mathbb{C}\setminus J^{\ast }$.
\end{proof}

\begin{figure}[H]
\begin{tikzpicture}[scale=0.7]
    \begin{scope}[thick,font=\scriptsize]


    \end{scope}
   \path [draw=none,fill=gray,semitransparent] (-5,-5) rectangle (5,5);
    \path [draw=none,fill=white] (0,0) circle (0.1);
   \path [draw=none,fill=white] (-0.05,-3) rectangle (0.05,-2);
   \path [draw=none,fill=white] (-0.05,3) rectangle (0.05,2);
    \draw (-3pt,2) -- (3pt,2)   node [right] {$i\sqrt{S^{\ast}_D}$};
    \draw (-3pt,3) -- (3pt,3)   node [right] {$i\sqrt{S_D}$};
    \draw (-3pt,-2) -- (3pt,-2)   node [ right] {$-i\sqrt{S^{\ast}_D}$};
    \draw (-3pt,-3) -- (3pt,-3)   node [right] {$-i\sqrt{S_D}$};
    \draw [->] (-5,0) -- (5,0) node [below left]  {$\Re \lambda$};
    \draw [->] (0,-5) -- (0,5) node [below right] {$\Im \lambda$};
    \draw [] (-5,0) -- (-2,0);
    \node [below right,gray] at (+3,+4) {$\mathbb{C}\setminus J^{\ast }$};
\end{tikzpicture}
\caption{The domain $\mathbb{C}\setminus J^{\ast }$.}
\label{figCminusJ}
\end{figure}
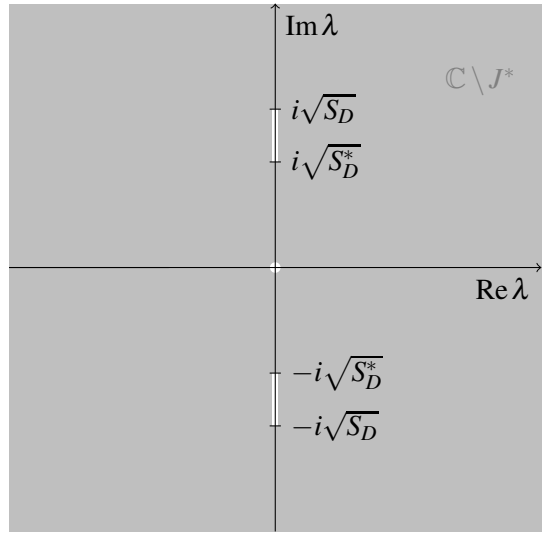

\section{Effective impedance of infinite networks}

\label{SecInfinite}

Let $(V,E)$ be an infinite locally finite connected graph equipped with the
weights $\rho _{xy}^{\left( \lambda \right) }$ as in Section \ref{SecFinite}%
. Fix a vertex $a\in V$ and a set of vertices $B\in V$ such that $%
a\not\in B$. Note that here the set $B$ can be empty (which physically
means that the ground will be at infinity).

Then the structure $\Gamma =(V,\rho ,a,B)$ is called an \emph{infinite
network}.

Let $\dist\left( x,y\right) $ be the graph distance on $V$, that is, the
minimal value of $n$ such that there exists a path $\left\{ x_{k}\right\}
_{k=0}^{n}$ connecting $x$ and $y$, that is,%
\begin{equation*}
x=x_{0}\sim x_{1}\sim ...\sim x_{n}=y.
\end{equation*}
Let us consider a sequence of finite graphs $(V_{n},E_{n})$, $n\in \mathbb{N}
$, where 
\begin{equation*}
V_{n}=\{x\in V\mid \dist(a,x)\leq n\}
\end{equation*}%
and $E_{n}$ consists of all the edges of $E$ with the endpoints in $V_{n}$.
We endow the finite graph $\left( V_{n},E_{n}\right) $ with the complex
weight $\rho _{n}=\left. \rho \right\vert _{E_{n}}.$

Consider the set 
\begin{equation*}
\partial V_{n}=\{x\in V\mid \dist(a,x)=n\},
\end{equation*}%
that will be regarded as the \emph{boundary} of the graph $(V_{n},E_{n})$.
Note that $V_{n}=\partial V_{n}\cup V_{n-1}$. Let us set%
\begin{equation*}
B_{n}=\left( B\cap V_{n}\right) \cup \partial V_{n}
\end{equation*}%
and consider the following sequence of finite networks 
\begin{equation*}
\Gamma _{n}=(V_{n},\rho _{n},a,B_{n}),\ \ n\in \mathbb{N}.
\end{equation*}

Let $\mathcal{P}_{n}(\lambda )$ be the effective admittance of $\Gamma _{n}$.

\begin{definition}
Define the \emph{effective admittance} of $\Gamma$ as 
\begin{equation*}
\mathcal{P}(\lambda)=\lim_{n\rightarrow\infty}\mathcal{P}_n(\lambda)
\end{equation*}
for those $\lambda \in \mathbb{C}\setminus \{0\}$ where the limit exists.
\end{definition}

\begin{theorem}
 \label{convergenceAtInfinity}

The following is true for any infinite network.
\begin{enumerate}

\item[$\left( a\right) $] The sequence $\{\mathcal{P}_{n}(\lambda )\}$
converges as $n\rightarrow \infty $ locally uniformly in the domain $\{\Re\lambda >0\}$.

\item[$\left( b\right) $] If $L_{xy}>0$ for all $xy\in E$ and 
\begin{equation}
S_{D}:=\sup_{xy\in E}\frac{D_{xy}}{L_{xy}}<\infty \;\text{and}%
\;S_{R}:=\sup_{xy\in E}\frac{R_{xy}}{L_{xy}}<\infty
\label{SdSrConvAtInfinity}
\end{equation}%
then $\{\mathcal{P}_{n}(\lambda )\}$ converges as $n\rightarrow \infty $
locally uniformly in the domain 
\begin{equation}
\Omega =\{\lambda \in \mathbb{C}\mid \Im\lambda \neq 0\ \text{and\ \ }%
|\lambda |^{2}>S_{D}\}.
\end{equation}

\item[$\left( c\right) $] If in addition to (\ref{SdSrConvAtInfinity}) also 
\begin{equation*}
S_{D}^{\ast }:=\inf_{xy\in E}\frac{D_{xy}}{L_{xy}}>0
\end{equation*}%
then $\{\mathcal{P}_{n}(\lambda )\}$ converges as $n\rightarrow \infty $
locally uniformly in the domain 
\begin{equation*}
\Omega ^{\ast }=\{\lambda \in \mathbb{C}\mid \Im\lambda \neq 0\ \ 
\text{and\ \ }|\lambda |^{2}<S_{D}^{\ast }\}.
\end{equation*}
\end{enumerate}

In all the cases, the limit 
\begin{equation*}
\mathcal{P}(\lambda)=\lim_{n\rightarrow\infty}\mathcal{P}_n(\lambda)
\end{equation*}
is a holomorphic function in the domains in question.
\end{theorem}

\begin{proof}
$(a)$ By Corollary \ref{corollaryBoundPlambdaC}, the sequence $\{\mathcal{P}%
_{n}(\lambda )\}$ is uniformly bounded in any domain 
\begin{equation*}
\{\Re\lambda \geq \epsilon ,\ |\lambda |\leq c\}
\end{equation*}%
with $0<\epsilon <c<\infty $. Hence, the sequence $\{\mathcal{P}_{n}(\lambda
)\}$ is precompact in such a domain and, hence, has a convergent
subsequence. By a diagonal process, we obtain a convergent subsequence $\{%
\mathcal{P}_{n_{k}}(\lambda )\}$ in the entire domain $\{\Re\lambda
>0\}$, and the limit is a holomorphic function in this domain. On the other
hand, for positive real $\lambda $ also all $\rho _{xy}^{\left( \lambda
\right) }$ are real and positive on the edges, and in this case the sequence 
$\{\mathcal{P}_{n}(\lambda )\}$ is known to be positive and decreasing (from
the theory of random walks on graphs, see e.g. \cite{Grimmett}, \cite{Soardi}%
). Hence, this sequence has a limit for all positive real $\lambda $. Since
every holomorphic function in $\{\Re\lambda >0\}$ is uniquely
determined by its values on positive reals, we obtain that $\lim \mathcal{P}%
_{n_{k}}(\lambda )$ is independent of the choice of a subsequence. Hence,
the entire sequence $\{\mathcal{P}_{n}(\lambda )\}$ converges as $%
n\rightarrow \infty $ in the domain $\{\Re\lambda >0\}$, and the limit
is a holomorphic function in this domain.

$(b)$ By Theorem \ref{HolomorphicDomainLargeIm} all the functions $\{\mathcal{P}%
_{n}(\lambda )\}$ are holomorphic in the domain 
\begin{equation*}
\Omega =\{\lambda \in \mathbb{C}\mid \Im\lambda \neq 0\ \ \text{and\ \ 
}|\lambda |^{2}>S_{D}\}
\end{equation*}%
and admit the estimate 
\begin{equation*}
\left\vert \mathcal{P}_{n}(\lambda )\right\vert \leq \frac{C(2|\lambda
|+S_{R})^{2}|\lambda |^{6}}{(|\lambda |^{2}-S_{D})^{3}|\Im\lambda |^{3}%
}.
\end{equation*}%
Hence, the sequence $\{\mathcal{P}_{n}(\lambda )\}$ is locally uniformly
bounded in $\Omega $ and, hence, is precompact. All the limits of convergent
subsequences of $\{\mathcal{P}_{n}(\lambda )\}$ coincide by $\left( a\right) 
$ in the domain 
\begin{equation*}
\Omega \cap \{\Re\lambda >0\}
\end{equation*}%
which implies that they coincide also in $\Omega $. Hence, $\{\mathcal{P}%
_{n}(\lambda )\}$ converges in $\Omega $ to a holomorphic function.

$(c)$ By Theorem \ref{HolomorphicDomainSmallIm} all functions $\{\mathcal{P}%
_{n}(\lambda )\}$ are holomorphic 
\begin{equation*}
\Omega ^{\ast }=\{\lambda \in \mathbb{C}\mid \Im\lambda \neq 0\ \text{%
and\ }|\lambda |^{2}<S_{D}^{\ast }\}
\end{equation*}%
and admit the estimate 
\begin{equation*}
\left\vert \mathcal{P}_{n}(\lambda )\right\vert \leq \frac{C^{\prime
}(|\lambda |^{2}+S_{R}|\lambda |+S_{D})^{2}|\lambda |^{4}}{(S_{D}^{\ast
}-|\lambda |^{2})^{3}|\Im\lambda |^{3}}.
\end{equation*}%
Hence, the sequence $\{\mathcal{P}_{n}(\lambda )\}$ is locally uniformly
bounded in $\Omega ^{\ast }$ and, hence, is precompact. All the limits of
convergent subsequences of $\{\mathcal{P}_{n}(\lambda )\}$ coincide in the
domain 
\begin{equation*}
\Omega ^{\ast }\cap \{\Re\lambda >0\}
\end{equation*}%
which implies that they coincide also in $\Omega ^{\ast }$. Hence, $\{%
\mathcal{P}_{n}(\lambda )\}$ converges in $\Omega ^{\ast }$ to a holomorphic
function.
\end{proof}

\begin{corollary}
\label{corollaryInfinity} Assume that $R_{xy}=0$ for all $xy\in E.$ Then $%
\mathcal{P}(\lambda )=\lim_{n\rightarrow \infty }\mathcal{P}_{n}(\lambda )$
is well-defined and holomorphic in the domain $\mathbb{C}\setminus J^{\ast }$%
, where 
\begin{equation*}
J^{\ast }=\left[ i\sqrt{S_{D}^{\ast }},i\sqrt{S_{D}}\right] \cup \left[ -i%
\sqrt{S_{D}},i\sqrt{S_{D}^{\ast }}\right] \cup \{0\}.
\end{equation*}
\end{corollary}

\begin{proof}
Note that we assume neither $S_{D}<\infty $ nor $S_{D}^{\ast }>0.$ By the
symmetry $\mathcal{P}(-\lambda )=-\mathcal{P}(\lambda )$ an by Theorem \ref%
{convergenceAtInfinity}, the sequence $\{\mathcal{P}_{n}(\lambda )\}$
converges locally uniformly in the union 
\begin{equation*}
\{\Re\lambda \neq 0\}\cup \{(\Im\lambda )^{2}>S_{D}\}\cup \{0<(%
\Im\lambda )^{2}<S_{D}^{\ast }\}
\end{equation*}%
that coincides with $\mathbb{C}\setminus J^{\ast }$.
\end{proof}

The next statement is a simplified version of Corollary \ref%
{corollaryInfinity}.

\begin{corollary}
\label{corinf}Assume that $R_{xy}=0$ for all $xy\in E$ and set 
\begin{equation*}
S:=\sup_{xy\in E}\frac{1}{C_{xy}L_{xy}}.
\end{equation*}%
Then $\mathcal{P}(\lambda )=\lim_{n\rightarrow \infty }\mathcal{P}%
_{n}(\lambda )$ is well-defined and holomorphic in the domain 
\begin{equation*}
\mathbb{C}\setminus \left[ -i\sqrt{S},i\sqrt{S}\right] .
\end{equation*}
\end{corollary}

\section{Examples}

\label{SecEx}

\begin{example}
\label{exLine} Consider the infinite graph $(V,E)$, where 
\begin{equation*}
V=\{0,1,2,\dots \}
\end{equation*}%
and $E$ is given by 
\begin{equation*}
0\sim 1\sim 2\sim \dots \sim n\sim (n+1)\sim \dots
\end{equation*}%
Define the impedance of the edge $k\sim (k+1)$ by 
\begin{equation*}
z_{k(k+1)}=L_{k}\lambda +\frac{D_{k}}{\lambda },
\end{equation*}%
where $L_{k}\geq D_{k}>0$ (and $R_{k}=0$) (see Fig. \ref{figChain}).

\begin{figure}[H]
\begin{circuitikz} 
  \draw
 (0,0) node[anchor=south]{$0$}
 (0,0) to[/tikz/circuitikz/bipoles/length=30pt, L,  l=${L_0}$, o-] (1.5,0)
 (1.5,0) to[/tikz/circuitikz/bipoles/length=20pt, C,  l=$D_0$, -*] (3,0)
 (3,0) node[anchor=north]{$1$}
 (3,0) to[/tikz/circuitikz/bipoles/length=30pt, L,  l=${L_1}$, *-] (4.5,0)
 (4.5,0) to[/tikz/circuitikz/bipoles/length=20pt, C,  l=$D_1$, -*] (6,0)
 (6,0) node[anchor=north]{$2$}
 (8,0) node[anchor=north]{$(n-1)$}
 (8,0) to[/tikz/circuitikz/bipoles/length=30pt, L,  l=${L_{n-1}}$, *-] (9.5,0)
 (9.5,0) to[/tikz/circuitikz/bipoles/length=20pt, C,  l=$D_{n-1}$, -*] (11,0)
 (11,0) node[anchor=north]{$n$};

  \draw[dashed] 
  (6,0) to[short, *-*] (8,0) 
  (11,0) to[short, *-] (13,0);
\end{circuitikz}
\caption{Chain network}
\label{figChain}
\end{figure}
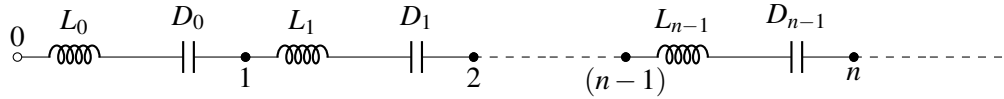

Set $a=0$ and $B=\emptyset $. Then we have $V_{n}=\left\{
0,...,n\right\} $ and $B_{n}=\{n\}$. It follows that 
\begin{equation*}
\mathcal{P}_{n}(\lambda )=\frac{1}{\sum_{k=0}^{n-1}\left( L_{k}\lambda +%
\frac{D_{k}}{\lambda }\right) }=\frac{1}{l_{n}\lambda +\frac{d_{n}}{\lambda }%
},
\end{equation*}%
where 
\begin{equation*}
l_{n}=\sum_{k=0}^{n-1}L_{k}\;\ \text{and}\;d_{n}=\sum_{k=0}^{n-1}D_{k}.
\end{equation*}%
Assume further that 
\begin{equation*}
\sum_{k=0}^{\infty }L_{k}=\sum_{k=0}^{\infty }D_{k}=\infty
\end{equation*}%
that is 
\begin{equation*}
\lim_{n\rightarrow \infty }l_{n}=\lim_{n\rightarrow \infty }d_{n}=+\infty .
\end{equation*}%
Then for any $\lambda $ with $\Re\lambda \neq 0$ we obtain 
\begin{equation*}
\Re\left( l_{n}\lambda +\frac{d_{n}}{\lambda }\right) \rightarrow
\infty
\end{equation*}%
whence 
\begin{equation*}
\mathcal{P}(\lambda )=\lim_{n\rightarrow \infty }\mathcal{P}_{n}(\lambda )=0.
\end{equation*}%
For $\lambda =i\omega $ with real $\omega $ we have 
\begin{equation*}
\mathcal{P}_{n}(i\omega )=-\frac{i}{l_{n}\omega -\frac{d_{n}}{\omega }}.
\end{equation*}%
Assume in addition that 
\begin{equation*}
\sum_{k=0}^{\infty }(L_{k}-D_{k})=:c\in (0,\infty )
\end{equation*}%
that is 
\begin{equation*}
\lim_{n\rightarrow \infty }(l_{n}-d_{n})=c\in (0,\infty ).
\end{equation*}%
Then for $\omega =1$ we have 
\begin{equation*}
\mathcal{P}_{n}(i)=-\frac{i}{l_{n}-{d_{n}}},
\end{equation*}%
whence 
\begin{equation*}
\mathcal{P}(i)=\lim_{n\rightarrow \infty }\mathcal{P}_{n}(i)=-\frac{i}{c}.
\end{equation*}%
It follows that also 
\begin{equation*}
\mathcal{P}(-i)=\frac{i}{c}.
\end{equation*}%
For $\omega \neq \pm 1$ we have 
\begin{equation*}
l_{n}\omega -\frac{d_{n}}{\omega }=(l_{n}-d_{n})\omega +d_{n}\left( \omega -%
\frac{1}{\omega }\right) \rightarrow \pm \infty
\end{equation*}%
whence it follows that 
\begin{equation*}
\mathcal{P}(i\omega )=\lim_{n\rightarrow \infty }\mathcal{P}_{n}(i\omega )=0.
\end{equation*}%
Hence, for any $\lambda \in \mathbb{C}\setminus \{0\}$ we have 
\begin{equation*}
\mathcal{P}(\lambda )=%
\begin{cases}
-\frac{\lambda }{c},\ \text{if}\ \lambda =\pm i \\ 
0,\ \text{otherwise.}%
\end{cases}%
\end{equation*}%
In particular, $\mathcal{P}(\lambda )$ is holomorphic in $\mathbb{C}%
\setminus \{\pm i\}$ but is discontinuous at $\lambda =\pm i$.

On the other hand we have 
\begin{equation*}
S_{D}=\sup_{n}\frac{D_{n}}{L_{n}}\ \ \text{and}\ \ S_{D}^{\ast }=\inf_{n}\frac{%
D_{n}}{L_{n}}.
\end{equation*}%
Since $c<\infty $ then necessarily $S_{D}=1$. Since $c>0$ then $S_{D}^{\ast
}<1$.

Wee see that the points $\lambda =\pm i$ (where $\mathcal{P}$ looses
continuity) lie in the set 
\begin{equation*}
J^{\ast }=\left[ -i\sqrt{S_{D}},-i\sqrt{S_D^{\ast }}\right] \cup \left[ i\sqrt{%
S_D^{\ast }},i\sqrt{S_{D}}\right] \cup \{0\}
\end{equation*}%
that matches Corollary \ref{corollaryInfinity}. For example, choose 
\begin{equation*}
L_{n}=1\;\text{and}\;D_{n}=1-\epsilon 2^{-n},
\end{equation*}%
where $\epsilon >0$. Then $S_{D}=1$ and $S_{D}^{\ast }=1-\epsilon $, so that
the interval 
$$\left[ i\sqrt{S_D^{\ast }},i\sqrt{S}_{D}\right]=\left[i\sqrt{1-\varepsilon},i\right], $$ 
containing $\lambda=i,$ can have an arbitrary small length.
\end{example}

\begin{example}
\label{exLadder}

Consider the infinite graph $(V,E)$, where 
\begin{equation*}
V=\{0,1,2,3,4,\dots \}
\end{equation*}%
and $E$ is given by $(2k-2)\sim 2k$ and $(2k-1)\sim 2k$ for $k=\overline{%
1,\infty }$ . Let us make this graph into a network as on  Fig. \ref{fig}. 

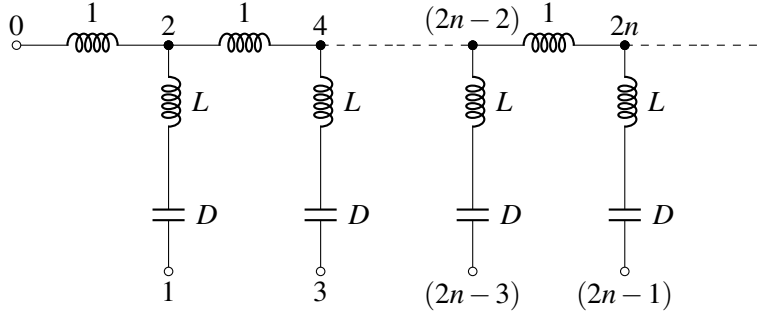
\begin{figure}[H]
\begin{circuitikz} 
  \draw
 (0,0) node[anchor=south]{$0$}
  (0,0) to[/tikz/circuitikz/bipoles/length=30pt,  L, l=$1$, o-*] (2,0)
  (2,0) node[anchor=south]{$2$}
  (2,0) to[/tikz/circuitikz/bipoles/length=30pt, L,  l=$1$, *-*] (4,0)
  (4,0) node[anchor=south]{$4$}
  (2,0) to[/tikz/circuitikz/bipoles/length=30pt, L,  l=$L$, *-] (2,-1.5)
  (2,-1.5) to[/tikz/circuitikz/bipoles/length=20pt,C,  l=$D$,-o] (2,-3)
  (2,-3) node[anchor=north]{$1$}
  (4,0) to[/tikz/circuitikz/bipoles/length=30pt, L,  l=$L$, *-] (4,-1.5)
  (4,-1.5) to[/tikz/circuitikz/bipoles/length=20pt,C,  l=$D$,-o] (4,-3)
  (4,-3) node[anchor=north]{$3$}

  (6,0) node[anchor=south]{$(2n-2)$}
  (6,0) to[/tikz/circuitikz/bipoles/length=30pt, L,  l=$1$, *-*] (8,0)
  (8,0) node[anchor=south]{$2n$}
   (6,0) to[/tikz/circuitikz/bipoles/length=30pt, L,  l=$L$, *-] (6,-1.5)
  (6,-1.5) to[/tikz/circuitikz/bipoles/length=20pt,C,  l=$D$,-o] (6,-3)
  (6,-3) node[anchor=north]{$(2n-3)$}
   (8,0) to[/tikz/circuitikz/bipoles/length=30pt, L,  l=$L$, *-] (8,-1.5)
  (8,-1.5) to[/tikz/circuitikz/bipoles/length=20pt,C,  l=$D$,-o] (8,-3)
  (8,-3) node[anchor=north]{$(2n-1)$};

  \draw[dashed] 
  (4,0) to[short, *-*] (6,0) 
  (8,0) to[short, *-] (10,0);
\end{circuitikz}
\caption{Modified ladder network}
\label{fig}
\end{figure}

That is, let the impedance of the edges $(2k-2)\sim 2k$ be $\lambda $ and impedance
of the edges $2k-1\sim 2k$ be $L\lambda +\frac{D}{\lambda }$, where $L>0$
and $D>0$. Set also $a=0$, while 
\begin{equation*}
B=\{1,3,\dots \}.
\end{equation*}%
This network is an $\alpha\beta$-network from \cite{Muranova2} and it is similar to Feynman's ladder network (see \cite{Feynman}),
but we add coils to the \textquotedblleft vertical\textquotedblright\ edges
and ground at infinity. Clearly, we have 
\begin{equation*}
V_{n}=\{0,1,\dots ,2n\}\setminus \{2n-1\}\ \ \text{and} \ \ B_{n}=\{1,3,\dots ,2n-3\}\cup \{2n\}.
\end{equation*}%
The Dirichlet problem (\ref{Dirpr}) for the finite network $\Gamma_n$ is as follows:
\begin{equation}\label{dirprF}
 \begin{cases}
 v(2k-2)+\mu v(2k-1)+ v(2k+2)-(2+\mu)v(2k)=0,\ \ k=\overline{1,n-1},
\\
v(0)=1,
\\
v(2k-1)= 0,\ \ k=\overline{1,n-1}, 
\\ 
v(2n)=0,
 \end{cases}
\end{equation}
where $\mu=\frac{\lambda^2}{L\lambda^2+D}$.

Substituting the equations from the third line of (\ref{dirprF}) to the first line and denoting $v_k=v(2k)$, we obtain the following recurrence relation for $v_k$:
\begin{equation}\label{recuk}
v_{k+1}-\left(2+\mu\right)v_k+v_{k-1}=0.
\end{equation}
The characteristic polynomial of (\ref{recuk}) is
\begin{equation}\label{psi}
\psi^2-\left(2+\mu\right)\psi+1=0.
\end{equation}
By the definition of a network $\mu\ne 0$. If $\mu\ne -4$, then the equation (\ref{psi}) has two different complex roots $\psi_1, \psi_2$ and its solution is
\begin{equation}\label{ukeq}
v_k=c_1\psi_1^k+c_2\psi_2^k,
\end{equation}
where $c_1,c_2\in \mathbb{C}$ are arbitrary constants.

We use the second and fours equations of (\ref{dirprF}) as boundary conditions for this recurrence equation.
Substituting (\ref{ukeq}) in the boundary conditions we obtain the following equations for the constants:
\begin{equation*}
\begin{cases}
c_1+c_2=1,\\
c_1\psi_1^{n}+c_2\psi_2^{n}=0.
\end{cases}
\end{equation*}
Therefore,
\begin{equation*}
\begin{cases}
c_1=\frac{1}{1-\psi_1^{2n}}=\frac{-\psi_2^{2n}}{1-\psi_2^{2n}},\\
c_2=\frac{1}{1-\psi_2^{2n}}=\frac{-\psi_1^{2n}}{1-\psi_1^{2n}},
\end{cases}
\end{equation*}
since $\psi_1\psi_2=1$ by (\ref{psi}).

Now we can calculate the effective admittance of $\Gamma_n$:
\begin{align*}
\mathcal{P}_{n}(\lambda )=&{\frac{1}{\lambda}\left(1-v(2)\right)}=\frac{1}{\lambda}\left(1-c_1 \psi_1-c_2 \psi_2\right)\\
=&\frac{\left(\psi_1^{2n-1}+1\right)\left(\psi_1-1\right)}{\lambda\left(\psi_1^{2n}-1\right)}=\frac{\left(\psi_2^{2n-1}+1\right)\left(\psi_2-1\right)}{\lambda\left(\psi_2^{2n}-1\right)}.
\end{align*}%

Without loss of generality we can assume, that $|\psi_1|\le |\psi_2|$. Then,
since $\psi_1\psi_2=1$ by (\ref{psi}), we have either $|\psi_1|<1<|\psi_2|$ or $|\psi_1|=|\psi_2|=1$.

In the case $|\psi_1|<1<|\psi_2|$ we obtain
\begin{equation*}
\mathcal{P}(\lambda )=\lim_{n\rightarrow \infty }\mathcal{P}_{n}(\lambda )=\frac{1-\psi_1}{\lambda}.
\end{equation*}%

In the case $|\psi_1|=|\psi_2|=1$ the sequence $\{\mathcal{P}_{n}(\lambda )\}$ has no limit.

Let us now consider the case $\mu =-4$. Then the solution of the recurrence relation (\ref{recuk}) is 
\begin{equation*}
v_k=c_1(-1)^k+c_2k(-1)^k,
\end{equation*}
where $c_1,c_2\in \mathbb{C}$ are arbitrary constants. And using boundary conditions, we obtain
\begin{equation*}
\begin{cases}
c_1=1\\
c_2=-\frac{1}{n},
\end{cases}
\end{equation*}

\begin{equation*}
\mathcal{P}_{n}(\lambda )={\frac{1}{\lambda}\left(1-v(2)\right)}=\frac{2n-1}{\lambda n},
\end{equation*}%

and 

\begin{equation*}
\mathcal{P}(\lambda )=\lim_{n\rightarrow \infty }\mathcal{P}_{n}(\lambda )=\frac{2}{\lambda}.
\end{equation*}

Therefore, for the infinite network we have 
\begin{equation*}
\mathcal{P}(\lambda )=%
\begin{cases}
\frac{1-\psi_1}{\lambda}, \mbox { if }|\psi_1|<1<|\psi_2|,\\
\frac{2}{\lambda },\ \ \ \ \ \ \mbox{if } \psi_1=\psi_2=-1,\\
\mbox {not defined otherwise}. 
\end{cases}
\end{equation*}

Now we will reformulate the above identity in terms of $\lambda$.

\begin{claim} \label{modulus_x_1}
Let $\mu\ne -4$ and $\mu \ne 0$. Then the condition $|\psi_{1}|=|\psi_{2}|=1$ occurs if and only if $%
\mu \in (-4,0)$.
\end{claim}

\begin{proof}
\textquotedblleft $\Rightarrow $\textquotedblright\ Let $|\psi_{1}|=1$. Then $%
|\psi_{2}|=1$. Since $\mu \neq 0$
and $\mu \neq -4$, it follows from (\ref{psi}) and $%
|\psi_{1}|=|\psi_{2}|=1$ that $\psi_{1},\psi_{2}\not\in \mathbb{R}$. Therefore, $\psi_{2}=%
\overline{\psi_{1}}$, since $\psi_{1}\overline{\psi_{1}}=1$. Moreover, by (\ref%
{psi}) we have $2+\mu =\psi_{1}+\psi_{2}\in \mathbb{R}$. Also $%
\psi_{1},\psi_{2}\not\in \mathbb{R}$ means that the determinant of (\ref%
{psi}) 
\begin{equation}
\left( 2+\mu \right) ^{2}-4=4\mu +\mu ^{2}
\end{equation}%
is not positive, i.e. $\mu \not\in (-\infty ,-4)\cup (0,\infty )$.
Therefore, 
\begin{equation*}
\mu \in \mathbb{R}\setminus ((-\infty ,-4)\cup (0,\infty ))
\end{equation*}
which was to be proved.

\textquotedblleft $\Leftarrow $\textquotedblright\ Let $\mu \in (-4,0)$.
Then the determinant of (\ref{psi}) is negative and 
\begin{equation*}
|\psi_{1,2}|^{2}=\left\vert 1+\frac{\mu }{2}\pm i\sqrt{-\mu -\left( \frac{%
\mu }{2}\right) ^{2}}\right\vert ^{2}=\left( 1+\frac{\mu }{2}\right)
^{2}-\mu -\left( \frac{\mu }{2}\right) ^{2}=1.
\end{equation*}
\end{proof}

Since $\mu=-4$ at the points $\lambda =\pm i\sqrt{\frac{D}{L+1/4}}$, we have 

\begin{equation}\label{Pla}
\mathcal{P}(\lambda )=%
\begin{cases}
\frac{1-\psi_1}{\lambda }, \  \mbox{if}\ \lambda
\in \mathbb{C}\setminus \left[ -i\sqrt{\frac{D}{L+1/4}},i\sqrt{\frac{D}{L+1/4%
}}\right] \\ 
\frac{2}{\lambda },\ \ \ \ \ \ \mbox{if }\lambda =\pm i\sqrt{\frac{D}{L+1/4}} \\ 
\mbox{not defined, if }\lambda \in \left( -i\sqrt{\frac{D}{L+1/4}},i\sqrt{%
\frac{D}{L+1/4}}\right) ,%
\end{cases}
\end{equation}%
where $\psi_1$ is the root if the equation \eqref{psi} with $|\psi|<1$.
Clearly, this function is continuous at the points $\lambda =\pm i\sqrt{%
\frac{D}{L+1/4}}$. Indeed, 
\begin{equation*}
\psi\rightarrow -1,
\end{equation*}%
when $\lambda \rightarrow \pm i\sqrt{\frac{D}{L+1/4}}$, since roots of the
quadratic equation are continuous functions on coefficients. Hence, the
continuity of $\mathcal{P}(\lambda)$ at given points follows. Therefore, $\mathcal{P}\left( \lambda \right)$ is well-defined and continuous in the domain
\begin{equation*} 
\mathbb{C}\setminus \left( -i\sqrt{\frac{D}{L+1/4}},i\sqrt{\frac{D}{L+1/4%
}}\right).
\end{equation*} 
In particular, the domain of holomorphicity of $\mathcal{P}\left( \lambda \right)$ is 
\begin{equation} \mathbb{C}\setminus \left[ -i\sqrt{\frac{D}{L+1/4}},i\sqrt{\frac{D}{L+1/4%
}}\right].\label{DL4}
\end{equation}

On the other hand we have $S_{D}^{\ast }=0$, $S_{D}=\frac{D}{L}$, therefore, 
the Corollary \ref{corollaryInfinity} states, that $\mathcal{P}\left( \lambda \right)$ is holomorphic in the domain
\begin{equation} 
\label{DL}
\mathbb{C}\setminus \left[ -i\sqrt{\frac{D}{L}},i\sqrt{\frac{D}{L}}\right].
\end{equation} 
Comparison of the intervals (\ref{DL4}) and (\ref{DL}) shows the sharpness of Corollary \ref{corollaryInfinity}.  
\end{example}

\section*{Acknowledgement}
The author thanks her scientific advisor, Professor Alexander Grigor'yan, for helpful comments related to this work.


\begin{thebibliography}{99}
\bibitem{Brune} O. Brune. \emph{Synthesis of a finite two-terminal network
whose driving-point impedance is a prescribed function of frequency. Thesis
(Sc. D.)}. Massachusetts Institute of Technology, Dept. of Electrical
Engineering. Massachusetts, 1931.

\bibitem{DS} P.G. Doyle, J.L. Snell. \emph{Random walks and electric networks%
}. Carus Mathematical Monographs 22, Mathematical Association of America.
Washington, DC, 1984.

\bibitem{Enk} S. J. van Enk.  \emph {Paradoxical   behavior   of   an   infinite   ladder   network   of   inductors   and   capacitors}. {American Journal of Physics}, 2000. Vol 68, n.9, 854--856.

\bibitem{Feynman1} Richard P. Feynman, Robert B. Leighton, Matthew Sands. 
\emph{The Feynman lectures on physics, Volume 1: Mainly mechanics,
radiation, and heat}. Addison-Wesley publishing company. Reading,
Massachusetts, Fourth printing -- 1966.

\bibitem{Feynman} Richard P. Feynman, Robert B. Leighton, Matthew Sands. 
\emph{The Feynman lectures on physics, Volume 2: Mainly Electromagnetism and
Matter}. Addison-Wesley publishing company. Reading, Massachusetts, Fourth
printing -- 1966.

\bibitem{Grigoryan} A. Grigor'yan. \emph{Introduction to Analysis on Graphs}%
. AMS University Lecture Series, Volume: 71. Providence, Rhode Island, 2018.

\bibitem{Grimmett} G. Grimmett. \emph{Probability on Graphs: Random
Processes on Graphs and Lattices}. Cambridge University Press. New York,
2010.

\bibitem{Klimo}Paul Klimo. \emph{On the impedance of infinite $LC$ ladder networks}. {%
European journal of physics}, 2016. Vol. 38, n. 1, 1--9.
\bibitem{LPW} David A. Levin, Yuval Peres, Elizabeth L. Wilmer. \emph{Markov
Chains and Mixing Times}. AMS University Lecture Series. Providence, Rhode
Island, 2009.

\bibitem{Muranova} Anna Muranova. \emph{On the notion of effective impedance}%
. arXiv e-prints, page arXiv:1905.02047, May 2019.

\bibitem{Muranova2} Anna Muranova. \emph{Effective impedance over ordered
fields}. arXiv e-prints, page arXiv:1907.13239, July 2019.

\bibitem{Soardi} Paolo M. Soardi. \emph{Potential Theory on Infinite Networks%
}. Springer-Verlag, Berlin Heidelberg, 1994.

\bibitem{UA} C. Ucak and C. Acar. \emph{Convergence and periodic solutions for the input impedance of a standard ladder network.} {%
European journal of physics}, 2007. Vol. 28, n. 2, 321--329.

\bibitem{UY} C. Ucak and K. Yegin. \emph{Understanding the behaviour of infinite ladder circuits.} {%
European journal of physics}, 2008. Vol. 29, n. 6, 1201--1209.


\bibitem{Vasquez} Fernando Guevara Vasquez, Travis G. Draper, Justin Tse,
Toren E. Wallengren, Kenneth Zheng. \emph{Matrix valued inverse problems on
graphs with application to elastodynamic networks}.
https://arxiv.org/abs/1806.07046, June 2018. 

\bibitem{Woess} Wolfgang Woess. \emph{Random Walks on Infinite Graphs and
Groups}. Cambridge Tracts in Mathematics: 138. Cambridge University Press,
2000.

\bibitem{Yoon} Sung Hyun Yoon. \emph{Ladder-type circuits revisited}. {%
European journal of physics}, 2007. Vol. 22, n. 22, 277--288.
\end{thebibliography}
\end{document}